\documentclass[fullpage,11pt]{amsart}
\usepackage{amssymb,verbatim,graphicx}
\usepackage{tikz}
\usetikzlibrary{matrix, arrows}
\usepackage{graphics,epsfig,psfrag} 
\newtheorem{theorem}{Theorem}[section]
\newtheorem{proposition}[theorem]{Proposition}
\newtheorem{corollary}[theorem]{Corollary}
\newtheorem{lemma}[theorem]{Lemma}
\theoremstyle{definition}
\newtheorem{definition}[theorem]{Definition}
\theoremstyle{remark}
\newtheorem{remark}[theorem]{Remark}
\newtheorem{example}[theorem]{Example}

\newcommand{\cL}{\mathcal{L}}

\newcommand{\R}{\mathbb{R}}
\newcommand{\C}{\mathbb{C}}
\newcommand{\eps}{\epsilon}

\newcommand{\cO}{\mathcal{O}}
\newcommand{\cU}{\mathcal{U}}

\newcommand{\bsh}{\backslash}
\newcommand{\ra}{\rightarrow}

\newcommand{\pa}{\partial}
\begin{document}
\title[Complement of Normal Crossing Divisor]{On the complement of a positive normal crossing divisor with no triple intersection in a projective variety}
\author{Khoa Lu Nguyen}
\address{Department of Mathematics,
UC Berkeley, California 94720.
{\em E-mail address: nglukhoa@berkeley.edu}}
\thanks{This research was funded by NSF grants DMS-1344991 and DMS-1205349.}

\begin{abstract} We describe how the Weinstein structure of the complement is modified when a positive divisor degenerates to a normal crossing divisor with no triple intersection.
\end{abstract}

\maketitle

\section{Introduction}
Let $(X, J)$ be a complex manifold and $\cL \rightarrow X$ be a holomorphic line bundle. In complex geometry (\cite{GH}), $\cL$ is \textbf{positive} if it admits a metric $g$ such that the curvature $\omega \in H^{1,1}(X)$ (with respect to the natural unique Hermitian connection $\nabla_g$) is a Kahler form. When $X$ is a projective variety, positivity of $\cL$ is equivalent to \textbf{ampleness}. A line bundle $\cL$ is ample if some positive tensorial power $\cL^{\otimes n}$ is very ample, i.e. $\cL^{\otimes n}$ is generated by its holomorphic sections. 

Let $s \in H^0(X, \cL)$ be a holomorphic section of $\cL$. The zero set $D_s := s^{-1}(0)$ of $s$ is a divisor of complex codimension 1 in $X$. The complement $W_{s} := X \bsh D_s$ inherits a \textbf{Weinstein structure} $(\lambda_s, \phi_s)$ from $(\cL, g, s)$ (Definition \ref{wmanifold}):  
$$\phi_s := -\mbox{log}(\|s\|^2), \hspace{0.3 cm} \lambda_s = -d^{c}\phi_s = -J^*d\phi_s.$$

The Weinstein structure $(W_{s}, \lambda_s, \phi_s)$ satisfies $d\lambda_s = \omega.$ Notice that the 1-form $\lambda$ blows up at $D_s$ although its exterior derivative $\omega$ extends to $X$. In fact, since $J$ is integrable, this structure has a classical name in complex geometry world: Stein structure. It is important to stress that the Weinstein structure is an object in the symplectic world because the integrable complex structure is not part of the data.

The complement $W_s$ topologically depends on $s$. For those sections $s$ whose zero sets $D_s$ are smooth divisors, $W_s$ are all diffeomorphic. In fact, the Weinstein structures $(W_s, \lambda_s, \phi_s)$ are all Weinstein homotopic (Definition \ref{whomotopy}). 

\textbf{Question:} Suppose $\{s_{\epsilon}\}_{\epsilon \in \mathbb{R}}$ is a family of sections of $\cL$ such that $D_\epsilon := D_{s_{\epsilon}}$ is smooth except when $\epsilon = 0$, $D_0$ is singular. How do we describe the modification of the Weinstein structure of $W_0$ in term of $W_\epsilon?$

This paper studies the case $D_0$ is a \textbf{normal crossing divisor} with no triple intersection. By definition, a divisor is normal crossing if there exists local holomorphic coordinates $(z_1, ..., z_n)$ near each point $p \in D_0$ such that $D_0$ is given by the equation $z_1...z_k = 0$ for some $1 \leq k \leq n$. $D_0$ has no triple intersection if and only if $k \leq 2$. 

\begin{remark} In symplectic topology, there are parallel notions of symplectic divisor and symplectic normal crossing divisor (\cite{Don}, \cite{MTZ}). Although we make some use of the integrable complex structure in the proof of the main result, there should be similar results.
\end{remark} 

In the case $\cL$ is very ample, Bertini's theorem \cite{GH} guarantees the existence of $\{s_\epsilon\}$ near $s_0$. Without loss of generality, assume $s_\epsilon = s_0 + \epsilon h,$ where $h \in H^0(X, \cL).$ The singular set $S$ of $D_0$ is a smooth complex codimension 2 submanifold. Let $B$ be the base locus of the pencil generated by $\{s_\eps\}$ and denote $\bar{S} = S \bsh (S \cap B).$ The (imprecise) main result is the following (for precise statements, see Theorem \ref{1} and Theorem \ref{2}). 

\begin{theorem} \label{mainresult} Up to homotopy of Weinstein structures, there is a Weinstein (domain) embedding of $$i: (W_0, \lambda_0, \phi_0) \rightarrow (W_\epsilon, \lambda_\epsilon, \phi_\epsilon).$$ The complement $W_\epsilon \bsh i(W_0)$ is decomposed into elementary Weinstein cobordisms, whose critical points are corresponding to critical points of $\bar{S}$, with a shift of indices by 2. In the case the critical point in a cobordism has index $\mbox{dim}_{\mathbb{C}}X$, the (Lagrangian) unstable submanifold can be obtained from the fibration $\pi:= s_0/h:  X \bsh {h^{-1}(0)} \rightarrow \mathbb{C}$ and the (Lagrangian) unstable submanifold of the corresponding critical point in $\bar{S}$.            
\end{theorem}

We depart to discuss an application of Theorem \ref{mainresult}. There is a large class of Weinstein manifolds which appear from algebraic geometry: \textbf{affine varieties}. Given an affine variety $W \subset \mathbb{C}^N$, $W$ inherits a Weinstein structure from the distance function $\phi(z) = |z|^2.$ Hironaka's theorem implies that $W$ can be compactified to a projective variety $X$ by adding a normal crossing divisor $D$ at infinity. Such compactification is generally very useful in computing symplectic invariants $W$, such as its \textbf{symplectic homology}. 

Symplectic homology, originally defined and extended in \cite{FH}, \cite{V1}, \cite{V2}, is Hamiltonian-Floer theory applied to the non-compact setting. In the case of Weinstein manifold $(W, \lambda, \phi)$, roughly speaking, symplectic homology is the homology of a chain complex generated by critical points of $\phi$ (which record the topology of $W$) and Reeb orbits of the contact manifold $(\pa W, \lambda|_{\pa W})$. By choosing appropriate boundary $\pa W$ of $W$, the compactification $X = W \cup D$ gives a  nice description of the Reeb dynamics $\pa W$ as follows. There exists an open subset $U \subset \pa W$ such that $\pa W \bsh U$ is a circle bundle over the complement of a neighborhood of the singular set of $S$ in $D$, whose circle fiber is a Reeb orbit of $\pa W$. The homology classes of the Reeb orbits in $U$ can be written as linear combinations of the fibers in $\pa W \bsh U$. The differential of the symplectic homology chain complex can then be computed via     
some energy or topological arguments or via reducing the solutions of Cauchy-Riemann equations to some holomorphic curve counts (\cite{D}, \cite{P}).

Back to our situation, the first statement in Theorem \ref{mainresult} implies that there is a Viterbo transfer chain map $$\Phi: SH_*(W_\eps) \rightarrow SH_*(W_0).$$ In our forthcoming paper (\cite{N}), following the Bourgeois-Ekholm-Eliashberg's approach (\cite{BEE}),  we show that the mapping cone of $\Phi$ is quasi-isomorphic to a chain complex $LH_*$ constructed from the Lagrangian unstable submanifolds. This chain complex is parallel to the Hochschild chain complex from the $A_\infty$-structure of these Lagrangian unstable submanifolds in the wrapped Fukaya category of $W_\eps$ (\cite{Ab}, \cite{AbSei}). Since the Lagrangian unstable submanifolds is obtained from the fibration $\pi$ and some Lagrangian balls in $S$, it is reasonable to expect that one may reduce computing $LH_*$ to analyzing the Lagrangian balls in $S$ and the fibration $\pi$. This may be useful because $S$ has complex codimension 2 in $X$ making it easier to work with.  

\textbf{Acknowledgement}: This paper is the first part of the author's Ph.D thesis. The author would like to thank his advisor Yakov Eliashberg for many enlightening discussions throughout the years. The author also thank Mohammed Abouzaid, Denis Auroux and Tobias Ekholm for very helpful conversations about symplectic homology. 

\section{Weinstein structure}

We recall from \cite{CE} some basic notions of Weinstein structure. 

\begin{definition} \label{wmanifold} A \textit{Weinstein structure} on an open manifold $W$ consists of $(W, \lambda, \phi)$, where

\begin{itemize}  \item $d\lambda = \omega$ is a symplectic form on $W$,
\item $\phi: W \rightarrow \mathbb{R}$ is an exhausting generalized Morse function,
\item the vector field $Z$ given by $i_Z\omega = \lambda$ is complete and gradient-like for $\phi$.
\end{itemize}

In the case $\phi$ has only finitely many critical points, $(W, \lambda, \phi)$ is said to be of \textit{finite type}.
\end{definition}

By an exhausting generalized Morse function, we mean it is a proper and bounded from below function. Moreover, its critical points are either nondegenerate or embryonic. The latter means that in a neighborhood of the critical point $p$ with coordinates $(x_1, ..., x_m)$, $\phi = \phi_0$ is in the birth-death family: $$\phi_{t}(x) = \phi_t(p) \pm tx_1 + x_1^3 - \sum_{i=2}^{k} x_i^2 + \sum_{j=k+1}^{m} x_j^2.$$ 
The vector $Z$ is gradient-like with respect $\phi$ if there is some Riemannian metric on $W$ and a positive function $\delta: W \ra \mathbb{R}_+$ such that $$d\phi(Z) \geq \delta(|Z|^2 + |d\phi|^2).$$

\begin{definition} A \textit{Weinstein cobordism} $(W, \lambda, \phi)$ is a compact manifold $W$ with $\pa W = \pa_{-} W \cup \pa_{+}W$, which are regular level sets of the minimal and maximal values of $\phi$, and such that $(\lambda, \phi)$ satisfies all the conditions in Definition \ref{wmanifold} except that $Z$ is inwardly and outwardly transversal to $\pa_{-}W$ and $\pa_{+}W$ respectively. 

If $\pa_{-} W = \varnothing$, then $(W, \lambda, \phi)$ is called a \textit{Weinstein domain}.
\end{definition}

Given a Weinstein domain $(W, \lambda, \phi),$ we can attach a positive half of the symplectization $(\mathbb{R}_+ \times \pa W, e^{r}\lambda|_{\pa W}, f(r))$ (with an appropriate strictly increasing function $f$) to obtain a Weinstein manifold. This is called the completion of $W$, also denoted by $W$. Any finite type Weinstein manifold is the completion of some Weinstein domain. 

Our main source of examples comes from affine varieties. 

\begin{example} If $W \subset \mathbb{C}^N$ is an affine variety, $W$ admits a Weinstein structure $\lambda = -d^{c}|z|^2$ and $\phi$ some $C^\infty$-small perturbation of $|z|^2$. 
\end{example}

\begin{example} Suppose $X$ is a projective variety with a positive line bundle $\cL.$ The complement $W_s$ of a normal crossing divisor $D_s = s^{-1}(0)$ of a holomorphic section $s$ of $\mathcal{L}$ admits  a Morse function (after some perturbation) $\phi_s(z) = -\mbox{log}\|s(z)\|^2.$ Define $\lambda_s = -d^{c}\phi_s = -J^*d\phi_s.$ Any sublevel set $(\{\phi_s \leq M\}, \lambda_s, \phi_s)$ with regular value $M$ is a Weinstein domain. The set of critical points of $\phi$ forms a compact subset in $W_s$ (\cite{Sei}). When $M$ is sufficiently large so that its complement in $W_s$ contains no critical point of $\phi_s$, denote by $W_s$ the Weinstein manifold obtained from completing $(\{\phi_s \leq M\}, \lambda_s, \phi_s)$. 
\end{example}

\begin{definition} \label{whomotopy} A \textit{Weinstein homotopy} of finite type Weinstein structures $\{(W, \lambda_{t}, \phi_t)\}_{t \in [0,1]}$ is a family such that $\cup_{t \in [0,1]}\mbox{Crit}(\phi_t)$ is compact in $W$ and the underlying $(W, \phi_{t})$ is a Morse homotopy.

\end{definition}

Recall that $(W, \phi_{t})$ (with $\cup_{t \in [0,1]}\mbox{Crit}(\phi_t)$ is compact in $W$) is a Morse homotopy if  there exists a finite set $A \subset (0,1)$ such that for each $t \in A$, the function $\phi_t$ has a unique birth-death type critical point such that it does not lie on any other critical levels. For $t \notin A$, $\phi_t$ is Morse.  

\section{The Geometry of the complement}
Recall the setting of our problem: there is a family of holomorphic sections $s_{\epsilon} = s_0 + \epsilon h$ such that the zero sets $D_\eps$ are smooth except at $\eps = 0$, where $D_0$ is a normal crossing divisor with no triple intersection. 

Denote by $B = \{s_0 = h = 0\}$ the base locus of this family. The restriction of $\phi_\eps$ to $\bar{S} = S \bsh (S \cap B)$, where $S \subset D_0$ is the singular set of $D_0$, is independent of $\eps$ up to constant addition. The differential $d\phi_\eps|_{\bar{S}}$ is thus independent of $\eps$ and we define $$\mbox{Crit}(\bar{S}) := \mbox{Crit}(\phi_\eps|_{\bar{S}}).$$    

Assume $\mbox{Crit}(\bar{S})$ is a discrete set.

\begin{theorem} \label{1}  Up to homotopy of Weinstein structures, there is a Weinstein embedding of $$i: (W_0, \lambda_0, \phi_0) \rightarrow (W_\epsilon, \lambda_\epsilon, \phi_\epsilon).$$ For $\eps > 0$ sufficiently small, there is a natural bijection between the set of critical points of $\phi_\eps$ in the complement $W_\epsilon \bsh i(W_0)$ and $\mbox{Crit}(\bar{S}).$ Moreover, the index of a critical point equals the index of the corresponding element in $\mbox{Crit}(\bar{S})$ plus 2.
\end{theorem}

\begin{proof}
Let $K \subset W_0$ be a sublevel of $\phi_0$ such that $\mbox{Crit}(\phi_0) \subset K$. For sufficiently small $\eps$, the Weinstein vector field $Z_\eps$ is $C^{1}$-close to $Z_0$. Hence, by choosing a sublevel $K_\eps \subset W_\eps$ such that $K \subset K_\eps$, we obtain an embedding $i: (K, \lambda_0, \phi_0) \ra (K_\eps, \lambda_\eps, \phi_\eps),$ proving the first half of the theorem.

The critical points of $\phi_\eps$ in $K_\eps$ converge to critical points of $\phi_0$ as $\eps \ra 0$. Let's analyze other critical points of $\phi_\eps$. We now show that given any neighborhood $U$ of $S$, the other critical points of $\phi_\eps$ lie in $U$ for sufficiently small $\eps$. Assume the contrary that there is a sequence $\{p_{m} \in \mbox{Crit}(\phi_{\eps_m})\}$ such that as $\eps_m \ra 0$, $p_{\eps_m} \ra p \notin U.$ Then $p \in D_0 \bsh S$. There is a local holomorphic coordinates $(z_0, \mathbf{z})$ near $p$ and a local trivialization $''1''$ of $\cL$ such that $s_\eps = (z_0 + \eps h)''1''.$ Denote $g(z_0, \mathbf{z}) := -\mbox{log}\|''1''\|^2,$ we have at $p_{m}$, $$\frac{\partial  \phi_{\eps_m}}{\partial z_0}(p_{m}) = \frac{\partial g}{\partial z_0}(p_{m}) - \frac{1}{z_0 + \eps_m h(z)} = 0.$$ As $\eps_m \ra 0$, it contradicts finiteness of $\partial g/{\partial z_0}(p)$.   

Moreover, the other critical points of $\phi_\eps$ can not wander off to $B$. Indeed, assume there is a sequence $\{p_{m} \in \mbox{Crit}(\phi_{\eps_m})\}$ such that as $\eps_m \ra 0$, $p_{m} \ra p \in B.$ There is a local holomorphic coordinates $(z_0, z_1, z_2, \mathbf{z})$ and a local holomorphic trivialization $''1''$ of $\cL$ such that $s_\eps = (z_0z_1 + \eps z_2)''1''.$ Again, denote $g = -\mbox{log}\|''1''\|^2,$ we have at $p_m$
$$\frac{\partial \phi_{\eps_m}}{\partial z_0} = \frac{\partial g}{\partial z_0} -\frac {z_{1}}{z_{0}z_1 + \eps_m z_2} = 0,$$
$$\frac{\partial \phi_{\eps_m}}{\partial z_{1}} = \frac{\partial g}{\partial z_{1}} - \frac {z_{0}}{z_{0}z_1 + \eps_m z_2} = 0,$$
$$\frac{\partial \phi_{\eps_m}}{\partial z_{2}} = \frac{\partial g}{\partial z_{2}} - \frac {\eps_m}{z_{0}z_1 + \eps z_2} = 0. $$
 The first two equations imply that there exists $C > 0$ such that for sufficient small $\eps_m$, at $p_m$, $\mbox{max}(|z_0|, |z_1|) \leq C\eps_m|z_2|.$ Thus 
$$\left|\frac{\partial g}{\partial z_{2}}(p_m)\right| = \left|\frac {\eps_m}{z_{0}z_1 + \eps z_2} \right| \geq \frac{1}{|z_2|} . \frac{1}{1 + |C^2\eps_m z_2|}, $$ contradicting finiteness of $\partial g/{\partial z_2}(p)$.

Having checked that the other critical points of $\phi_\eps$ must stay near $\bar{S}$ and away from $B$, we proceed to prove the second half of the theorem. It is obvious that if a sequence $\{p_{m} \in \mbox{Crit}(\phi_{\eps_m})\}$ converges to $p \in \bar{S}$, then $p \in \mbox{Crit}(\bar{S})$. Conversely, we will show that for each $p \in \mbox{Crit}(\bar{S}),$ any sufficiently small neighborhood $U$ of $p$ satisfies $\phi_\eps|_{U}$ has a unique critical point for sufficiently small $\eps.$ Let $(z_0, z_1, \mathbf{z})$ be local holomorphic coordinates near $p$ and local holomorphic trivialization $ ''1''$ of $\cL$ such that $s_\eps = (z_0z_1 + \eps) ''1''$. Denote by $g = -\mbox{log}\| ''1''\|^2.$ Let's consider the function $$F(z_{0}, z_1,\mathbf{z}, \eps) = ( z_1 - (z_{0}z_1 + \eps)\frac{\partial g}{\partial z_{0}}, z_0 -(z_{0}z_1 + \eps) \frac{\partial g}{\partial z_1},  -\frac{\partial g}{\partial \mathbf{z}}, \eps).$$
Since the derivative
$$DF(p,0) = 
\begin{bmatrix}
 0 & 0 & 1 & 0 & * & 0\\
0 & 0 & 0 & 1 & * & 0 \\
1 &  0 & 0 & 0 & * & 0 \\
0 & 1 & 0 & 0 & * & 0 \\
0 & 0 &0 &0 &  \mbox{Hess}_{x,y}(-g)(p) & 0\\

* & * & * & * & 0
 & 1\\
\end{bmatrix}
$$
is nondegenerate, Inverse Function Theorem implies that $F$ is a local diffeomorphism. In particular, for each sufficiently small $\eps$, there exists a unique solution $(p_\eps, \eps)$ near $(p, 0)$ satisfying $F(p_\eps, \eps) = (0, 0, 0, \eps)$, i.e. $p_\eps$ is a critical point of $\phi_\eps$.  

Finally, we need to show that $\mbox{index}(p_\eps) = \mbox{index}(p) + 2$. Computing the second derivatives of $\phi_\eps$,
we obtain all second derivatives of $\phi_\eps$ at $p_\eps$ are bounded except for $$\frac{\partial^2 \phi_\eps}{\partial x_0\partial x_1} \hspace{0.2 cm} \mbox{and} \hspace{0.2 cm}  \frac{\partial^2 \phi_\eps}{\partial y_0\partial y_1},$$ which have order of $\eps^{-1}$. This implies that as $\eps \ra 0$, there are two eigenvalues of $\mbox{Hess}(\phi_\eps)(p_\eps)$ diverging to $+\infty$, two eigenvalues diverging to $-\infty$, and the other eigenvalues converging to the eigenvalues of $\mbox{Hess}(\phi_\eps|_{\bar{S}})(p)$ (which are independent of $\eps$).Therefore, $\mbox{index}(p_\eps)  = \mbox{index}(p).$
\end{proof}
 
\begin{example} \label{ex1}
Let $X = \C P^n$ with coordinates $[z_0 : z_1 : ... : z_n]$ and $[H]$ is the hyperplane class. Take $\cL = \cO (2H)$ and equip $\cL$ with the following metric: on the affine part $\{z_i \neq 0\}$, the norm of the section $z^2_i$ in $\cO (2H)$ is $|z_i|^2/|z|^2$. Consider $$s_{0} = z_0z_1 \hspace{0.2 cm} \mbox{and} \hspace{0.2 cm} h = a_2z_2^2 + ... + a_nz_n^2,$$ where $a_2, ..., a_n \in \mathbb{C}$ are generic. Then $W_\eps$ and $W_0$ are Weinstein equivalent to $T^*\R P^n$ and $\C^* \times \C ^{n-1}.$ To see the first, consider the hypersurface $F = \{s_\eps + z_{n+1}^2 = 0\} \subset \C P^{n+1}$. The Weinstein structure on $F \bsh \{z_{n+1} = 0\}$ can be identified with a Weinstein structure on $T^*{S}^{n}$, which is homotopic to the standard Weinstein structure on $T^*{S}^n$ (\cite{Sei1}). There is a free $\mathbb{Z}_2$-action on $F \bsh \{z_{n+1} = 0\}$: $$[z_0 : ... : z_n : z_{n+1}] \ra [z_0: ... : z_n: -z_{n+1}],$$ the quotient of which can be identified with $W_\eps$. The corresponding $\mathbb{Z}_2$-action on $T^*{S}^n$ gives the standard Weinstein structure on $T^*{\mathbb{R}P}^n.$

The critical point set of $\phi_0$ is $\{[z_0 : ... : z_n] | |z_0| = |z_1|\}.$ The critical point set of $\phi_\eps$ is the union of the critical point set of $\phi_0$ and $$\mbox{Crit}{\bar{S}} = \{[0 :  0 : 1 : ... : 0], ..., [0: 0: 0: ... : 1]\}.$$  
\end{example}

\begin{example} \label{ex2} Let $X = \C P^n \times \C P^n$ with first and second factors ' coordinates $[z_0 : z_1 : ... : z_n]$ and $[z'_0 : z'_1 : ... : z'_n]$ and $[H]$ and $[H']$ are the hyperplane classes of the first and second factors. Take $\cL = \cO (H) \otimes\cO (H')$ and equip it with a similar metric described in the previous example. Consider $$s_0 = z_0z'_0 \hspace{0.2 cm}  \mbox{and} \hspace{0.2 cm} h = a_1z_1z'_1 + ... + a_nz_nz'_n,$$ where $a_1, ..., a_n \in \mathbb{C}$ are generic. Then $W_\eps$ and $W_0$ are Weinstein equivalent to $T^*\C P^n$ and $\C^{2n}.$ To see the first, consider the hypersurface $F = \{s_\eps + z_{n+1}z'_{n+1} = 0\} \in \mathbb{C}P^{n+1} \times \mathbb{C}P^{n+1}.$ The Weinstein structure on $F \bsh \{z_{n+1}z'_{n+1} = 0\}$ can be identified with a Weinstein structure on $T^{*}S^{2n+1},$ which is homotopic to the standard Weinstein structure on $T^*S^{2n+1}$. There is a free $\C^*$-action on $F$: $$t.([z_0 : ..., z_n : z_{n+1}], [z'_0 : ... : z'_n : z'_{n+1}]) = ([tz_0 : ... : tz_n : z_{n+1}], [t^{-1}z'_0 :  ... : t^{-1}z'_n : z'_{n+1}]),$$ the quotient of which can be identified with $W_\eps$. The corresponding $\mathbb{C}^*$-action on $T^*S^{2n+1}$ gives the standard Weinstein structure on $T^*\C P^n.$   

The function $\phi_0$ has only one critical point $([1: 0 : ... : 0], [1 : 0 : ... : 0])$, which has index 0. The set of critical points of  $\phi_\eps$  
consists of the critical point of $\phi_0$ and $n$ points $$\mbox{Crit}(\bar{S}) = \{([0 : 0 : ... : \underbrace{1}_\text{$i^{th}$} :  0  : ... : 0] \big| 1 \leq i \leq n\},$$ with critical indices $2, 4, ..., 2n.$  
\end{example}

\begin{example}
Consider two holomorphic sections $\tilde{s}_0$ and $\tilde{s}_1$ of some positive line bundle$\tilde{\cL} \ra X$ such that $\tilde{D}_0 = \tilde{s}_0^{-1}(0)$ is smooth. Define $\tilde{X} = X \times \mathbb{C}P^1$ with line bundle $\cL = \tilde{\cL} \otimes \mathcal{O}(H).$ Define $s_0 = z_0\tilde{s}_0$ and $h = z_1\tilde{s}_1$ and assume that $D_\eps$ is smooth except for $\eps = 0$. Notice that $D_0 = \tilde{D}_0 \times \mathbb{C}P^1 \cup X \times [0:1]$ is a normal crossing divisor with no triple intersection. Its singular set $S$ is $\tilde{D}_0 \times [0 : 1]$. The divisor $D_\eps$ is the blow up of $X$ along the base locus $\tilde{B} = \{\tilde{s}_0 = \tilde{s}_1 = 0\}$ of the pencil generated by $\tilde{s}_0$ and $\tilde{s}_1$.  

The complement $W_0$ is $X \times \mathbb{C}$ with subcritical Weinstein structure (i.e. the index of each critical point is strictly less than $\mbox{dim}_{\mathbb{C}} W_0$). All of the interesting symplectic topology lies in a neighborhood of $\bar{S} \simeq \tilde{D}_0 \bsh \tilde{B}$ in $\tilde{X}.$
\end{example} 

The following proposition is a generalization of Example \ref{ex1} and Example \ref{ex2}.

\begin{proposition} \label{prop1}
Let $E \subset TX|_{\bar{S}}$ be the $\omega$-orthogonal complement of $T\bar{S}$. Assume that the derivative of $g = -\mbox{log}\|h\|^2$ vanish in $E$. Then $\mbox{Crit}(\bar{S}) \subset \mbox{Crit}(\phi_\eps).$ Moreover, the vector field $Z_\eps$ is tangent to $\bar{S}$. As a consequence, the unstable submanifold of a critical point $p$ of $\phi_\eps$ which lies in $\mbox{Crit}(\bar{S})$ contains the unstable submanifold (lying in $\bar{S}$) of $p$ of $\phi_\eps|_{\bar{S}}.$
\end{proposition}
\begin{proof} The proof is straightforward from the following computation. In $\bar{S}$, $d\mbox{log}|s_0/h|^2 = 0.$ Thus since $dg|_{E} = 0$, $p \in \mbox{Crit}(\bar{S})$ implies that $d\phi_\eps(p) = 0,$ i.e. $p$ is a critical point of $\phi_\eps$. Since $\lambda_\eps = -d^{c}\phi_\eps = -d^{c}g$ vanish in $E$, $Z_\eps$ must be $(d\lambda_\eps = \omega)-$orthogonal to $E$. Hence $Z_\eps$ is tangent to $\bar{S}$. 
\end{proof}

The rest of the paper discusses about the unstable submanifold of $p_\eps \in \mbox{Crit}(\phi_\eps)$ which corresponds to $p \in \mbox{Crit}(\bar{S}).$ Under special assumption, Proposition \ref{prop1} implies that the unstable submanifold $\cU_p$ of $p_\eps$ is a disc bundle over the unstable submanifold $\tilde{\cU}_p$ lying in $\bar{S}$ of $p$.

\begin{figure*}[htp]
\centering
\includegraphics[width=\textwidth]{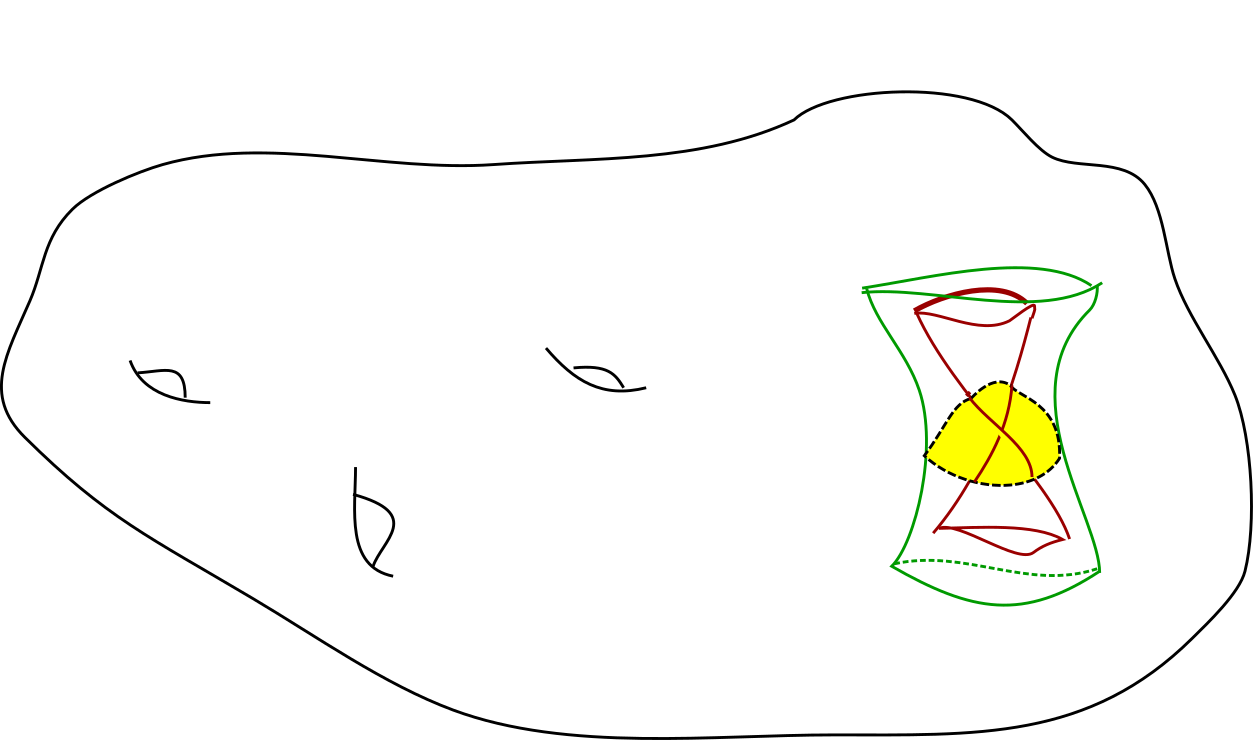}
\caption{The thimble $D_\eps(p)$ in the case $\mbox{dim}_\mathbb{C} X = 2.$}
\end{figure*}

\begin{example} Consider $\mathbb{C}^n$ with coordinates $(z_0, z_1, \mathbf{z})$ and $$\phi(z_0, z_1, z) = -\mbox{log}|z_0z_1 + 1|^2 + |z_0|^2 + |z_1|^2 + \psi(\mathbf{z})$$ a pluri-subharmonic function. If $p$ is a critical point of $\phi$, then $z_0(p) = z_1(p) = 0$ and $p$ is a critical point of $\psi$. Denote by $\tilde{\cU}_p \subset \{z_0 = z_1 = 0\}$ the unstable submanifold of $p$ of $\psi$. Then the unstable submanifold $\cU_p$ of $p$ of $\phi$ is $\cU_p = \{(z_0, z_1, \mathbf{z}) | z_0 = -\bar{z}_1, |z_0| \leq 1, \mathbf{z} \in \tilde{\cU}_p\}.$       
\end{example}

In $X \bsh h^{-1}(0)$, there is a fibration $\pi: X \bsh h^{-1}(0) \ra \mathbb{C},$ $\pi := s_0/h$. Denote by $$\eta_\eps := -d^c\mbox{log}|\pi + \eps|^2 \in H^1(X \bsh h^{-1}(0))$$ and vector field $X_\eps$ given by $i_{X_\eps}\omega = \eta_\eps$.  

Since we assume $\pi^{-1}(\eps)$ is smooth for small $\eps$, the fibration induces a connection in $X \bsh (h^{-1}(0) \cup S)$ as follows: at each point $z$, define $C_z \subset T_zX$ the subspace that is $\omega-$orthogonal to the fiber $\pi^{-1}(\pi(z))$. Consider $\gamma: [0, 1] \ra \mathbb{C}: \gamma(t) = t$. The connection induces a characteristic foliation on $\pi^{-1}((0, 1])$. For any compact subset $K \subset \bar{S}$, we denote by $D_\eps(K)$ the set of points in $\pi^{-1}((0, \eps])$ which under the characteristic flow converge to  points in $K$. Notice that $\{D_\eps(K)\}_\eps$ is a nested sequence: if $\eps_1 < \eps_2,$ $D_{\eps_1}(K) \subset D_{\eps_2}(K).$ This is a normal crossing generalization of the standard Lefschetz fibration picture in symplectic topology (\cite{Sei1}). 

The following lemma is crucial. 

\begin{lemma} \label{lefschetz}  For every $\epsilon$, $X_\eps$ is tangent to $\pi^{-1}(0, \eps)$ and lies along the characteristic foliation.

\end{lemma}
\begin{proof} Since $i_{X_\eps}\omega = \eta_\eps,$ $JX_\eps$ is tangent to the level sets $|\pi| = C$. In particular, $d\pi(JX_\eps)$ is tangent to the circle $|z| = C.$ Since $\pi$ is holomorphic, $d\pi(X_\eps)$ is perpendicular to the circle $|z| = C$. Thus, $X_\eps$ is tangent to $\pi^{-1}(0, \eps)$. 

By definition of $\eta_\eps$, $\eta_\eps(T\pi^{-1}(z)) = 0$. Hence, $X_\eps$ is $\omega-$orthogonal to every fiber. Therefore, $X_\eps$ must lie along the characteristic foliation.
\end{proof}

Denote by $\tilde{\cU}_p$ the unstable submanifold (lying in $\bar{S}$) of $p \in \mbox{Crit}(\bar{S})$. We have the following theorem.

\begin{theorem} \label{2}  Suppose $p$ has critical index, i.e. $\tilde{\mathcal{U}}_p$  is a Lagrangian in $\bar{S}$. Let $K \subset \tilde{\cU}_p$ be a compact ball containing $p$. Given a small neighborhood $V \subset X$ of $K$, then for each sufficiently small $\eps$, there exists a Weinstein homotopy that starts with $(W_\eps, \lambda_\eps, \phi_\eps)$ and stays constant outside $V$ such that the resulting Weinstein structure has in $V$ a unique critical point coinciding with $p$ and its unstable submanifold in $K$ is $D_\eps(K)$.   
\end{theorem}

\begin{proof} We first prove the case $\mbox{dim}_\mathbb{C}X = 2$. 

In complex two dimensional case, $S$ is a set of finitely many points. Let $p \in S$ and $K = \{p\}$. Denote by $D_\eps(p) := D_\eps(K)$, also called a thimble in the Lefschetz fibration $\pi$. Fix an $\eps_0 > 0$ (we will later make a choice of sufficiently small $\eps_0$). Since $D_{\eps_0}(p)$ is a Lagrangian disc, Weinstein Neighborhood Theorem implies its neighborhood is symplectomorphic to $T^*D^2.$  This gives coordinate $(u_1, v_1, u_2, v_2)$ in a neighborhood of $D_{\eps_0}(p)$ such that
\begin{itemize} 
\item $D_{\eps_0}(p) = \{u^2_1 + u^2_2  \leq \eps^2_0, v_1 = v_2 = 0\}$,
\item the symplectic form $\omega = du_1 \wedge dv_1 + du_2 \wedge dv_2$,
\item the complex structure $J$ satisfies $J(\pa /\pa u_1) = \pa /\pa v_1$, $J(\pa/\pa u_2) = \pa/\pa v_2$ on $D_{\eps_0}(p).$
\end{itemize}  

In this neighborhood, define $\xi_\eps: = \eta_\eps - v_1du_1 - v_2du_2$ and $\psi_\eps := \phi_\eps(p) -\mbox{log}|\pi + \eps|^2 + \frac{1}{2}(v_1^2 + v_2^2)$. Notice that $d\xi_\eps = \omega$. The Liouville vector field $Y_\eps$ of $\xi_\eps$ is given by $$Y_\eps = X_\eps + v_1\frac{\pa}{\pa v_1} + v_2\frac{\pa}{\pa v_2}.$$

Notice that Lemma \ref{lefschetz} implies that $D_\eps(p)$ is the unstable submanifold of $Y_\eps$. In fact, for sufficiently small $\delta > 0$, the neighborhood $V_\delta := \{u^2_1 + u^2_2 \leq \eps_0^2, |v_1|, |v_2| < \delta\}$ satisfies $(V_\delta \bsh D_\eps, \xi_\eps, \psi_\eps)$ is a Weinstein structure for all sufficiently small $\eps$. This follows from $-d^{c}\psi_\eps = \xi_\eps$ and $-dd^{c}\psi_\eps =\omega$ on $D_{\eps_0}(p)$. 

Since $d\lambda_\eps - d\xi_\eps = 0$ and $\lambda_\eps - \xi_\eps$ is a finite one form in $V_\delta$ (although both $\lambda_\eps$ and $\xi_\eps$ blow up in $D_\eps \cap V_\delta$), we can find $H: V_\delta \ra \mathbb{R}$ such that $$\lambda_\eps - \xi_\eps = dH \hspace{0.2 cm} \mbox{and} \hspace{0.2 cm} H(p) = 0.$$ 

Let $\rho: \mathbb{R} \ra \mathbb{R}$ be an increasing smooth function with compact support $[\eps_0/4, +\infty)$  such that $\rho|_{[\eps_0/2, +\infty]} = 1$ and $|\rho'| < 10\eps_0^{-1}$.
Denote $$r := u_1^2 + v_1^2 + u_2^2 + v_2^2$$ and define $$\tilde{\lambda}_\eps := (1-\rho(r))\xi_\eps + \rho(r)\lambda_\eps + Hd\rho.$$

This 1-form $\tilde{\lambda}_\eps$ is $\xi_\eps$ in $\{r < \eps_0/4\}$ and is $\lambda_\eps$ in $\{r > \eps_0/2\}.$ Moreover, it is a primitive of $\omega$ $$d\tilde{\lambda}_\eps = (1-\rho(r) + \rho(r))\omega + d\rho \wedge (-\xi_\eps + \lambda_\eps) + dH \wedge d\rho = \omega.$$

The following lemma is crucial.

\begin{lemma} \label{bound} $\tilde{\lambda}_\eps - \eta_\eps$ is $\eps_0$-uniformly bounded in $V_\delta$.
\end{lemma}
\begin{proof} Indeed, it is sufficient to check that $Hd\rho$ is $\eps_0$-uniformly bounded. Since $H(p) = 0$, there is $C > 0$ such that $$H(u,v) \leq C\sqrt{r}.$$ Also $C$ can be choosen so that $$|d\rho| = |\rho'(r)| |dr| \leq \frac{C\sqrt{r}}{\eps_0}.$$
Hence on $\{ \eps_0/4 < r < \eps_0/2\}$, we obtain $|Hd\rho| < C/2.$ Outside this region, by definition $Hd\rho = 0$. Thus, $Hd\rho$ is $\eps_0$-uniformly bounded.\end{proof}   

Lemma \ref{bound} implies that the Liouville vector field $\tilde{Z}_\eps$  of $\tilde{\lambda}_\eps$ (i.e. $i_{\tilde{Z}_\eps}\omega = \tilde{\lambda}_\eps$) satisfies $\tilde{Z}_\eps - X_\eps$ is $\eps_0$-uniformly bounded. Denote $$\tilde{\phi}_\eps := (1-\rho)\psi_\eps + \rho \phi_\eps.$$
We will show that $(W_\eps, \tilde{\lambda}_\eps, \tilde{\phi}_\eps)$ is a Weinstein structure when $\eps_0$ and $\eps$ are sufficiently small. In $\{r < \eps_0/4\},$ it is the Weinstein structure $(\xi_\eps, \psi_\eps)$ while in $\{ r> \eps_0/2\}$, it is the Weinstein structure $(\lambda_\eps, \phi_\eps)$. Consider the region $\{\eps_0/4 < r < \eps_0/2\}$. Write $$d\tilde{\phi}_\eps = (1-\rho)d\psi_\eps + \rho d\phi_\eps + (\phi_\eps - \psi_\eps)d\rho.$$
Since $\phi_\eps - \psi_\eps$ is a finite function in $V_\delta$ with value $0$ at $p$, a similar arguments as in Lemma \ref{bound} implies that $(\phi_\eps - \psi_\eps)d\rho$ is $\eps_0$-uniformly bounded. Hence, $d\tilde{\phi}_\eps + d\mbox{log}|\pi + \eps|^2$ is $\eps_0$-uniformly bounded. Now, $$d\mbox{log}|\pi + \eps|^2(\tilde{Z}_\eps) = d\mbox{log}|\pi + \eps|^2(X_\eps) + d\mbox{log}|\pi + \eps|^2(\tilde{Z}_\eps- X_\eps)$$ where the first term has order $\eps_0^{-1}$ and the later has order $\eps_0^{-1/2}$ if $\eps/\eps_0$ is sufficiently small . Therefore by choosing $\eps_0$ sufficiently small and $\eps$ such that $\eps/\eps_0$ sufficiently small, we obtain that $\tilde{Z}_\eps$ is gradient-like with respect to $\tilde{\phi}_\eps$. 

To construct a homotopy of Weinstein structure from $(W_\eps, \lambda_\eps, \phi_\eps)$ to $(W_\eps, \tilde{\lambda}_\eps, \tilde{\phi}_\eps)$, consider the path of 1-forms $(1-t)\lambda_\eps + t\tilde{\lambda}_\eps$ and functions $(1-t)\phi_\eps + t\tilde{\phi}_\eps$, where $t \in [0,1]$. A similar argument as before show that they form a Weinstein structure in $\{\eps_0/4 < r < \eps_0/2\}$. In the region $\{ r \leq \eps_0/4\}$, unfortunately they are not. For example, the zero of the vector field may not be a critical point of the function. However, a similar argument as in Theorem \ref{1} implies that for sufficiently small $\eps$, there is a unique zero of $(1-t)\lambda_\eps + t\tilde{\lambda}_\eps$ and a unique critical point of $(1-t)\phi_\eps + t\tilde{\phi}_\eps$ in $\{r \leq \eps_0 /4\}$. Since $-d^{c}\psi_\eps = \xi_\eps$ and $-dd^{c}\psi_\eps =\omega$ on $D_{\eps_0}(p)$, we also obtain $$-d^c[(1-t)\phi_\eps + t\tilde{\phi}_\eps] = (1-t)\lambda_\eps + t\tilde{\lambda}_\eps,$$ $$-dd^c[(1-t)\phi_\eps + t\tilde{\phi}_\eps] = d((1-t)\lambda_\eps + t\tilde{\lambda}_\eps) = \omega,$$ on $D_{\eps_0}(p)$.  Choosing $\eps_0$ sufficiently small, we can modify $(1-t)\phi_\eps + t\tilde{\phi}_\eps$ to make it a Weinstein structure with $(1-t)\lambda_\eps + t\tilde{\lambda}_\eps$ in $\{ r \leq \eps_0/4\}$.

Therefore in the case $\mbox{dim}_\mathbb{C} X = 2$, we obtain a Weinstein homotopy from $(W_\eps, \lambda_\eps, \phi_\eps)$ to $(W_\eps, \tilde{\lambda}_\eps, \tilde{\phi}_\eps)$ which satisfies: $\mbox{Crit}(\bar{S}) \subset \mbox{Crit}(\tilde{\phi}_\eps)$ and the unstable submanifold $p \in \mbox{Crit}(\bar{S})$ is $D_\eps(p)$. 

In higher dimensional case, the proof is similar. Let $K$ be a compact ball in $\bar{S}$ containing $p$. Since $K$ is a Lagrangian in $\bar{S}$ and $X_\eps$ preserves $\omega$, $D_{\eps_0}(K)$ is   Lagrangian in $X$. For sufficiently small $\eps_0$, $D_{\eps_0}(K)$ is a two-disc bundles over $K$.

 Choose coordinates $(u_1, v_1, u_2, v_2,  ... , u_n, v_n)$ in a neighborhood of $K$ so that
\begin{itemize}
\item  $S = \{u_1 = v_1 = u_2 = v_2 = 0\}$ and $D_{\eps_0}(K) = \{u^2_1 + u^2_2  \leq \eps^2_0, v_1 = v_2 = 0, (u_3, v_3, ...u_n, v_n) \in K\}$.
\item the symplectic form $\omega = du_1 \wedge dv_1 + du_2 \wedge dv_2 + ... + du_n \wedge dv_n$,
\item the complex structure $J$ satisfies $J(\pa /\pa u_i) = \pa /\pa v_i$ for $1 \leq i \leq n$ on $D_{\eps_0}(K).$
\end{itemize}

The coordinate defines a projection $\pi_K$ from the neighborhood to $\bar{S}$. Define $\xi_\eps: = \eta_\eps - v_1du_1 - v_2du_2 -  ... - v_ndu_n$ and $\psi_\eps := \phi_\eps \circ \pi_K -\mbox{log}|\pi + \eps|^2 + \frac{1}{2}(v_1^2 + v_2^2)$. Notice that $d\xi_\eps = \omega$. The Liouville vector field $Y_\eps$ of $\xi_\eps$ is given by $$Y_\eps =  Z_{S} + X_\eps + v_1\frac{\pa}{\pa v_1} + v_2\frac{\pa}{\pa v_2},$$ where $Z_S$ is the Weinstein vector field of $(\bar{S}, \lambda_\eps|_{\bar{S}}, \phi_\eps|_{\bar{S}})$ and is independent of $\eps$. 

Denote $V_\delta := \{u^2_1 + u^2_2 \leq \eps_0^2, |v_1|, |v_2| < \delta\}.$ Since $V_\delta$ is contractible, we can similarly construct $H: V_\delta \ra \mathbb{R}$ such that $\lambda_\eps - \xi_\eps = dH$ such that $H|_{V_\delta \cap \bar{S}} = 0.$ Let $f: \bar{S} \rightarrow \mathbb{R}$ such that $f$ has compact support in a neighborhood of $K$. Let $\rho$ be the cutoff function as before except that we also multiply it with $\pi^*_{K}f$. The 1-form $Hd\rho = 0$ along $V_\delta \cap \bar{S}$ and hence can be arbitrarily small along the horizontal directions , (i.e. the $(u_3, v_3, ..., u_n, v_n)-$ directions) if $\delta$ is sufficiently small. Also, as is proven in the lower dimensional case, $Hd\rho$ is $\eps_0$-uniformly bounded along the vertical directions (i.e. the $(u_1, v_1, u_2, v_2)-$ directions). This enables us to adapt the proof of the case $\mbox{dim}_{\mathbb{C}}X = 2$ to higher dimensional case.

\end{proof} 

The following corollaries are immediate. 

\begin{corollary} Given $K \subset \tilde{\cU}_p$ some sublevel in $\bar{S}$, for sufficiently small $\eps$, there is a Weinstein homotopy of $(W_\eps, \lambda_\eps, \phi_\eps)$ to $(W_\eps, \tilde{\lambda}_\eps, \tilde{\phi}_\eps)$ such that the unstable submanifold of $p$ in some regular sublevel $\{\tilde{\phi}_\eps \leq C\}$ is $D_\eps(K) \cap \{\tilde{\phi}_\eps \leq C\}$. 
\end{corollary}

\begin{corollary}  In the case $\mbox{dim}_{\mathbb{C}}X = 2,$ the unstable submanifold of $p$ in $(W_\eps, \tilde{\lambda}_\eps, \tilde{\phi}_\eps)$ is the thimble $D_\eps(p).$ The intersection of $D_\eps(p)$ with the divisor $D_\eps$ is an embedded Lagrangian circle $L_0$. 

Moreover, one can choose a boundary $Y$ of $(W_\eps, \tilde{\lambda}_\eps, \tilde{\phi}_\eps)$ (not necessarily a level set of $\tilde{\phi}_\eps$, but the vector field $\tilde{Z}_\eps$ is outwardly transversal to $Y$) so that the contact manifold $(Y, \tilde{\lambda}_\eps)$ is foliated by Reeb orbits which form a circle bundle over $D_\eps$. The intersection $\Gamma := D_\eps(p) \cap Y$ is a Legendrian knot, a lift of a Lagrangian $L_1$ in $D_\eps$, which is Hamiltonian isotopic to $L_0$, to the unit circle normal bundle of $D_\eps$ in $X$.
\end{corollary}
\begin{proof} The first half of the statement is obvious from the construction. Consider the normal bundle $\tilde{\pi}: \mathcal{N}D_\eps \ra D_\eps$ and equip it a connection $1-$form $\alpha$ so that the curvature of $\alpha$ is $\omega$ in $D_\eps$. Notice that $\alpha$ is a 1-form on $\mathcal{N}D_\eps \bsh D_\eps$, invariant under radial scaling and $d\alpha$ extends and equals $\omega$ in $D_\eps$. Symplectic Neighborhood Theorem yields a symplectomorphism $\Phi$ from a neighborhood of the zero section of $\mathcal{N}D_\eps$ (equipped with symplectic form $\Omega = d((1  - \frac{1}{2}r^2)\alpha)$) to a neighborhood of $D_\eps$ in $X$ such that $d\Phi = \mbox{Id}$ on $T(\mathcal{N}D_\eps)|_{D_\eps}$. Let $Y$ be the image of the circle bundle of radius $r_0$ under $\Phi$. 

Outside $V_\delta$ (where the Weinstein homotopy stays constant), it is well-known that $Z_\eps$ is outwardly transversal to $Y$. Inside $V_\delta$, from the construction, it is sufficient to check that the vector field $X_\eps$ is outwardly transversal to $Y$. Recall that $X_\eps$ lies in the $\omega$-orthogonal connection induced by the holomorphic map $\pi = s_0/h$. The vector field $d\pi(X_\eps)$ is perpendicular to circles with center $\eps$ in $\mathbb{C}$. The property of $\Phi$ implies that when $r_0$ is sufficiently small, $d\pi(X_\eps)$ is transversal to the image of each fiber of $Y$ under $\pi.$ Hence $X_\eps$ is outwardly transversal to $Y$ and consequently $\tilde{Z}_\eps$ is outwardly transversal to $Y$. 

Since $D_\eps(p)$ is Lagrangian and $\tilde{Z}_\eps$ is transversal to $Y$, the intersection $\Gamma = D_\eps(p) \cap Y$ is a Legendrian knot. This Legendrian intersects transversally each circle fiber at at most one point because the image $\pi(D_\eps(p)) \subset \mathbb{C}$ is the segment $[0, \eps]$. Flowing $\Gamma$ along $d\Phi(\frac{-1}{r(1-\frac{1}{2}r^2)}\pa/ \pa r)$ gives an embedded Lagrangian circle $L_1$. Similarly, for each $t \in [0,1]$, flowing $\Gamma$ along the vector field $(1-t)\tilde{Z}_\eps + td\Phi(\frac{-1}{r(1-\frac{1}{2}r^2)}\pa/ \pa r)$ gives an embedded Lagrangian circle $L_t$. 

We claim that $\{L_t\}_{0 \leq t \leq 1}$ are all Hamiltonian isotopic. Choose a bijective parametrization $\theta: S^1 \ra \Gamma$, we then get a parametrization $\theta_t: S^1 \ra L_t$. Denote by $\tilde{\theta}: [0,1] \times S^1 \ra D_\eps,$ $\tilde{\theta}(t, .) := \theta_t$. To show $\{L_t\}_{\{0 \leq t \leq 1\}}$ are all Hamiltonian isotopic, it is sufficient to check $$\int_{[0,t] \times S^1} \tilde{\theta}^*\omega = 0. $$ Let $C_t$ be the Lagrangian cylinder obtained from flowing $\Gamma$ along the vector field $(1-t)\tilde{Z}_\eps + td\Phi(\frac{-1}{r(1-\frac{1}{2}r^2)}\pa/ \pa r)$ (in particular, $\pa C_t = \Gamma \cup L_t$). We cap $\tilde{\theta}([0,t] \times S^1)$ with $C_0$ and $C_t$ to obtain a (torus) cycle $\mathcal{C}$ in $X$. In fact $\mathcal{C}$ lies in the complement of $D_0$, where $\omega$ is exact. It follows that $\int_{\mathcal{C}} \omega = 0$ and thus $\int_{[0,t] \times S^1} \tilde{\theta}^*\omega = 0$ holds. Therefore, $L_1$ is Hamiltonian isotopic to $L_0$.    
\end{proof}

\end{document}